\documentclass[reqno,11pt]{amsart}
\usepackage{graphicx} % Required for inserting images
\usepackage{amsthm}
\usepackage{amssymb}
\usepackage{amsmath}
\usepackage{bbm}
\usepackage{comment}
\usepackage[parfill]{parskip}
\usepackage{caption}
\usepackage{subcaption}
\usepackage{enumitem}
\usepackage[toc,page]{appendix}

\numberwithin{equation}{section}

\newcommand{\Probm}{\mathbb{P}}

\newcommand{\const}{\Tilde{c}}

\newcommand{\ConcavePotentialMeasure}{\hat{\Probm}}

\newtheorem{theorem}{Theorem}[section]
\newtheorem{prop}[theorem]{Proposition}
\newtheorem{definition}[theorem]{Definition}
\newtheorem{example}[theorem]{Example}
\newtheorem{lemma}[theorem]{Lemma}
\newtheorem{remark}[theorem]{Remark}

\newtheorem{assumption}[theorem]{Assumption}
\newtheorem{corollary}[theorem]{Corollary}

\newcommand{\bea}{\begin{eqnarray}}
\newcommand{\eea}{\end{eqnarray}}
\newcommand{\<}{\langle}
\renewcommand{\>}{\rangle}

\newcommand\eg{{\text{\eg~}}}

\def\eps{{\varepsilon}}

\def\de{{\rm d}}

\def\<{\langle}
\def\>{\rangle}

\def\b0{{\boldsymbol{0}}}

\renewcommand{\b}{\mathbf{b}}

\def\lt{\left}
\def\rt{\right}

\def\la{\langle}
\def\ra{\rangle}

\def\eps{\varepsilon}

\def\bbP{{\mathbb{P}}}
\def\bbR{{\mathbb{R}}}

\usepackage{xcolor}

\title[Brownian paths perturbed by pair potentials]{Mean square displacement of Brownian paths perturbed by bounded pair potentials}

\author[V. Betz]{Volker Betz}
\address{Volker Betz: TU Darmstadt}
\email{betz@mathematik.tu-darmstadt.de}
\author[T. Schmidt]{Tobias Schmidt}
\address{Tobias Schmidt: TU Darmstadt}
\email{tobias.schmidt@tu-darmstadt.de}
\author[M. Sellke]{Mark Sellke}
\address{Mark Sellke: Harvard University}
\email{msellke@fas.harvard.edu}

\begin{document}

\maketitle

\begin{abstract}
    We study Brownian paths perturbed by semibounded pair potentials and prove upper bounds on the mean square displacement. As a technical tool we derive
    infinite dimensional versions of key inequalities that were first used in \cite{Se22} in order to study the effective mass of the Fr\"ohlich polaron. 
\end{abstract}

\section{Introduction}
Perturbations of $d$-dimensional Brownian motion 
by a density that changes the behaviour of 
paths appear in many applications. A simple class of these measures is obtained by subjecting the Brownian paths to an external potential $V:\mathbb R^d \to \mathbb R$, resulting in probability measures of the type 
\[
\Probm_{V,T} (\de  x) = \frac{1}{Z_T} \exp \Big( 
- \int_0^T V(x_s) \, \de  s \Big) \Probm_{[0,T]} (\de x),
\]
where $\bbP_{[0,T]}$ is the Brownian path measure on $C \lt( [0,T]; \mathbb{R}^d \rt)$, and $Z_T$ is the normalization. Under suitable conditions on the function $V$ the Feynman-Kac
formula provides a link to semigroups of Schr\"odinger operators, and to the theory of It\^o-diffusions, and allows for an essentially complete understanding of these measures and their $T \to \infty$ limits. We refer to 
\cite{LBH-Book} for details. 

In this paper, we are concerned with the more difficult case where the perturbation appears through a pair potential $W$, meaning that we study 
probability measures of the form
\[
\ConcavePotentialMeasure_{\alpha,T} (\de x) = \frac{1}{\hat Z_T} \mathrm{e}^{\alpha \int_0 ^T \int_0 ^T W(\Vert x_t-x_s\Vert, |t-s| )  \de t\de s} \Probm_{[0,T]}(\de x).
\]
The potential acts on the Brownian increments, so that paths with increments which lead to large values of $W$ are favoured. 
$W$ usually decays to zero in the 
second argument fairly quickly, meaning that this influence is localized. We 
will be concerned with cases where $W(x,t)$ is maximal near $x=0$ for all $t$, 
which intuitively leads to a self-attractive force on Brownian paths: one 
would then expect that the mean square displacement 
$m_T(\alpha) = \ConcavePotentialMeasure_{\alpha,T}\lt( \Vert x_T \Vert^2 \rt)$ grows less quickly for large 
$\alpha$ 
than it does for small ones. 

Arguably the most important, and best studied case is where 
$W(x,t) = \frac{1}{|x|} {\mathrm e}^{-|t|}$ and the dimension is at least $3$, which corresponds to the 
Fr\"ohlich polaron, a basic model for matter interacting with a quantized 
field. We refer to \cite{DySp20} and the references therein for more background on these connections. 
A conjecture of Landau and Pekar \cite{LP48} 
states that in this case, $\lim_{T \to \infty} \frac{m_T(\alpha)}{T} = C_{\mathrm{LP}} \alpha^4 (1 + o(\alpha))$, 
with an explicit constant $C_{\mathrm{LP}}$. Proving this conjecture has been 
on the agenda of mathematical physicists at least since the seminal paper of Spohn \cite{Sp87}. However, 
significant progress was only made recently. First, 
using a point process representation of 
$\ConcavePotentialMeasure_{\alpha,T}$ introduced in \cite{MV19}, an upper 
bound of the order $\limsup_{T \to \infty} 
\frac{m_T(\alpha)}{T}  \leq C \alpha^{2/5}$ 
was shown in \cite{BP22b}, and then, using Gaussian domination, an almost 
optimal upper bound of the form $\limsup_{T \to \infty} 
\frac{m_T(\alpha)}{T} \leq C \alpha^4 / \log(\alpha)^6$ was 
shown in \cite{Se22}. In parallel, Brooks and Seiringer \cite{BrSe22} 
showed the sharp lower bound $\liminf_{T \to \infty} \frac{m_T(\alpha)}{T} \geq C_{\mathrm{LP}} \alpha^4 (1 + o(\alpha))$ by 
employing functional analytic methods on the quantum model of the polaron.  
So while the full Landau-Pekar conjecture is still not settled, we are now 
much closer to it than previously. 

The present paper builds on the techniques introduced in \cite{Se22} and has 
two main objectives: the first is concerned with the central tool introduced 
in \cite{Se22}, which is a way to leverage the Gaussian correlation 
inequality in order to obtain concentration inequalities for a rather large 
class of probability measures. In \cite{Se22}, these inequalities were 
introduced for finite dimensional measures, and their application to path 
measures was achieved via approximating the latter and by checking that the 
bounds obtained did not depend on the discretization parameter. 
Since these concentration inequalities may become a useful tool in various applications, 
it seems worth-while to develop a version that directly works in infinite 
dimensional spaces, and is thus directly applicable to path spaces. This is 
done in Theorems \ref{integra_ineq_prop} and \ref{produt_prop}. The second 
objective is to use these inequalities in order to get upper bounds on $m_T(\alpha)$
for potentials $W$ that are quite different from the Polaron case, in 
particular ones which are bounded from above and have compact support. This is the content of Theorem 
\ref{weaker_pot_thm}. The main conditions we require are that
%the class of potentials that we treat in this theorem is still much smaller than what we would like it to be: first of all, we need them to be of multiplicative form, $U(x,t) = W(x) f(t)$ with integrable $f$. More seriously,
$W$ must be 
symmetric and quasi-concave. Thus in many cases, fortunately including the Polaron, we now have powerful tools to estimate
the effect of pair potentials on Brownian paths.
However in several other cases the 
situation is just as unclear as it was before. In particular we currently 
have no tools to meaningfully estimate the effect of {\em repulsive} pair interactions. 

The paper is organized as follows: in Section 2, we introduce all relevant 
quantities and give precise statements of our results. In Section 3, we prove 
them. In an appendix, we collect and prove some facts about infinite dimensional measures that are most certainly known, but of which we were unable to find a satisfactory account in the literature.

\section{Definitions and Results}
Unless specified otherwise $\Omega$ is assumed to be a real, separable Banach space with Borel-$\sigma$-field $\mathcal{B}$.
A measurable function $f: \Omega \rightarrow \mathbb{R}$ is said to be symmetric whenever $f(x) = f(-x)$ for all $x \in \Omega$, and quasi-concave whenever  $\lt\{x: f(x) \geq r \rt\}$ is a convex set for all $r \in \mathbb{R}$.
\begin{definition}[QC]
     A function $f: \Omega \rightarrow \mathbb{R}$ is defined to be (QC) whenever $f$ is symmetric and quasi-concave. Furthermore, a function is also said to be (QC) whenever $f$ is a finite product of non-negative symmetric and quasi-concave functions or the uniformly bounded limit of such. 

     We say that a function $f: \Omega \rightarrow \mathbb{R}$ is (QC) on (a symmetric and convex set) $A$ whenever $\mathbbm{1}_A f$ is (QC).
 \end{definition}

The reader will have no problem to verify that the following functions are (QC):

        \begin{itemize}
            \item any symmetric function $m: \mathbb{R} \rightarrow \mathbb{R}$ which is decreasing on $\bbR_{\geq 0}$;
        \item $h \circ g$, where $g: \Omega \rightarrow \mathbb{R}$ is convex and symmetric and $h: \mathbb{R} \rightarrow \mathbb{R}$ is decreasing;
        \item $ \mathbbm{1}_K$, where $K \subseteq \Omega$ is a convex, symmetric set;
        \item $f \mathbbm{1}_K$, where $K \subseteq \Omega$ is a convex, symmetric set and $f \geq 0$ is (QC).
        \end{itemize}

\begin{example}
\label{integral_example}
    \normalfont
    Consider the special case $\Omega = C \lt( [0,1];\mathbb{R}^d \rt)$ and let $f \in C(\mathbb{R}^d;\mathbb{R})$ be bounded from above, symmetric and quasi-concave. The property (QC) is inherited by
    $$x \mapsto \mathrm{e}^{\int_0 ^1  f(x_t) \de t} = \lim\limits_{n \rightarrow \infty} \prod\limits^n _{j=1} \mathrm{e}^{\tfrac{1}{n} f(x_{j/n})}.$$
    This can be seen as follows: 
    $$x \mapsto \tfrac{1}{n} f(x_{j/n})$$ is (QC) per assumption. It follows that 
    $$x \mapsto \mathrm{e}^{\tfrac{1}{n} f(x_{j/n})}$$
    is (QC), as well. But then the entire product is uniformly bounded as the integral is bounded, and the claim is shown. The same reasoning lets us conclude that double integrals over increments are (QC), e.g. 
    $$x \mapsto \mathrm{e}^{\int_0 ^1 \int_0 ^1  f(x_t - x_s) \de s \de t}.$$
    
\end{example}

\begin{definition}
    \label{preceq_def}
    Let $\mu, \nu$ be two probability measures on $(\Omega, \mathcal{B})$. Write
    $$\nu \preceq \mu$$
    whenever
    $\frac{\de \nu}{\de \mu}$ is (QC).
\end{definition}

The significance of Definition \ref{preceq_def} stems from the following result.

\begin{prop}
Let $\mu, \nu$ be two probability measures on $(\Omega, \mathcal{B})$. Furthermore,
    let $\mu$ be a centred Gaussian measure and let $\nu \preceq \mu$.
    \begin{enumerate}
        \item For a (QC) function $f: \Omega \rightarrow \mathbb{R}$ with $f \geq 0$ or  $\Vert f \Vert_{\infty} < \infty$ 
    \begin{equation}
    \label{quasi_concave_inequ}
        \mu(f) \leq \nu(f).
        \end{equation}
        \item For $g:\Omega  \rightarrow \mathbb{R}_{\geq 0}$ symmetric and convex,
    \begin{equation}
    \label{convex_inequ}
        \mu(g) \geq \nu(g).
    \end{equation}
    \end{enumerate} 
\end{prop}

\begin{proof} 
     A functional version of the Gaussian correlation inequality (GCI, \cite{Ro14}) states that 
     for any centered Gaussian measure $\mu$ and any (QC) functions $f,g \geq 0$, we have $\mu(fg) \geq \mu(f) \mu(g)$; see \cite[Proposition 5.1]{BP22a}. This immediately gives 
     (\ref{quasi_concave_inequ}) for non-negative $f$:  
    \begin{equation}
        \label{divide_equ}
        \mu(f) \mu \lt(\frac{\de \nu}{\de \mu} \rt) \leq \mu \lt(f \frac{\de \nu}{\de \mu} \rt).
    \end{equation}
    For $f$ bounded and (QC) the claim follows by applying (\ref{quasi_concave_inequ}) to the non-negative (QC) function $f + \Vert f \Vert_{\infty}$.

    For (\ref{convex_inequ}) note that, for any $R>0$, $ - \min \{ R,g\}$ is (QC) as $\min(R,\cdot)$ is non-decreasing. Clearly  this function is bounded and so inequality (\ref{quasi_concave_inequ}) can be applied to obtain
    $$- \mu\lt( \min \{ R,g\} \rt) \leq  - \nu \lt( \min \{ R,g\}  \rt).$$ 
    An application of monotone convergence finishes the proof. 
\end{proof}

Define for a probability measure $\mu$ and for each Borel set $A \in \mathcal{B}$
$$\mu ^{\times 2} (A) := \mu (A/2).$$
The next Lemma is essential in the proofs of Theorem \ref{integra_ineq_prop} and Theorem \ref{produt_prop}. It provides us with a tool that makes it possible to work with a centred Gaussian measure $\mu$ on a finite-dimensional vector space as if it would be supported on a convex, symmetric set $K$ which has most of its mass. The important part is that the ``bad part'' of $\mu$, which is not supported on $K$, is still dominated by $\mu^{\times 2}$; this makes it possible to obtain easy error estimates. We note that the proof given in \cite{Se22} is not applicable whenever the support of $\mu$ is infinite-dimensional due to the fact that $\mu^{\times 2}$ and $\mu$ are singular; otherwise, the result of Theorem \ref{integra_ineq_prop} would be immediate.  

\begin{lemma}[{\cite[Lemma 3.1]{Se22}}]
\label{Se22Lemma31}
Let $\mu$ be a centred Gaussian measure on a finite-dimensional real vector-space $X$ and let $K \subseteq X$ be a symmetric convex set with $\mu(K) > 1- \delta$ for some $\delta \leq 0.1$. Then, there exists a decomposition 
$$\mu = (1-\delta') \nu + \delta' \overline{\nu}$$
of $\mu$ into a mixture of probability measures such that for some $c_1 \geq 100$:
\begin{enumerate}
    \item $\delta' \leq \delta$,
    \item supp($\nu$) $\subseteq c_1 K$,
    \item $\nu \preceq \mu$,
    \item $\overline{\nu} \preceq \mu^{\times 2}$.
\end{enumerate}
\end{lemma}

For a non-negative function $f$ with $\mu(f) < \infty$ we write
$$\mu^{(f)}(\de x) := \frac{f(x)}{\mu(f)}\mu(\de x).$$
We will use the notation $\mu \propto \nu$ for two 
positive measures if $\nu$ is a finite measure and $\mu = \tfrac{\nu}{\nu(\Omega)}$. In this notation, 
\[
\mu^{(f)}(\de x) \propto f(x) \mu(\de x).
\]
It is possible to approximate any centred, non-degenerate Gaussian measure $\mu$ by centred Gaussian measures with finite-dimensional support using Theorem \ref{approximation_theorem}. In what follows $(\mu_n)_{n \in \mathbb{N}}$ references this approximation sequence. It is now possible to state the first result:
\begin{theorem}
\label{integra_ineq_prop}
    Let $\mu$ be a non-degenerate, centred Gaussian measure and let $A \in \mathcal{B}$ be a closed, symmetric convex set with $\mu(A^c) < 0.1$ and $ \mu(\partial A)=0$.  
    \begin{enumerate}
        \item There exists a constant $c_1 <\const < c_1+1$ s.t. for any (QC) function $f \in C_b(\Omega;\mathbb{R}_{\geq 0})$ and $\delta \leq \mu(A^c)$,
    \begin{equation}
    \label{integral_inequ_1}
        \mu (f) \geq (1-\delta) \mu(f;\const A) + \delta \mu^{\times 2}(f).
    \end{equation}
    \item Take two (QC) functions $g \in C_b(\Omega;\mathbb{R}_{\geq 0})$ and $h \in C_b(\Omega;\mathbb{R}_{> 0})$. If, for all $n \in \mathbb{N}$, $\mu_n ^{(h)} \preceq \mu_n$ holds and $\mu_n ^{(h)}$ is still a centred Gaussian measure and $g/h$ is (QC) on $c_{1} A$, then for any (QC) function $f \in C_b(\Omega; \mathbb{R})$
    \begin{equation}
    \label{integral_inequ_2}
        \mu^{(g)} (f) \geq (1-\delta) \mu ^{(h)}(f) + \delta \mu^{\times 2}(f).
    \end{equation}   
    \end{enumerate}
\end{theorem}

Our next result is an extension of Theorem \ref{integra_ineq_prop} to the $T$-fold product of a centred Gaussian measure, $\Tilde{\mu} = \mu^{\otimes T}$. The strategy to prove Theorem \ref{produt_prop} will be to split each component of  $\mu^{\otimes T}$ using Lemma \ref{Se22Lemma31}. We then have to deal with $2^T$ product measures indexed by $\gamma \in \{0,1\}^T$,
where on each component the support differs. 

    Fix a closed, convex symmetric set $A \subseteq \mathcal{B}$.
    For $\gamma \in \{0,1\}^T$ define  $$A^{\gamma} := B_{\gamma(0)} \times ... \times B_{\gamma(T-1)},$$
    where  $B_0 = \Omega$, $B_1 = A$.
    Also,
    $$ S(\gamma):= \{ i < T-1: \gamma(i) = 1 = \gamma(i+1) \}$$
    and, for $j=0,1$,
    $$M_j(\gamma) :=  \{ i: \gamma(i) = j\}.$$

In the following, a vector $\mathbf{h}$ is defined such that finite dimensional approximations of $\mu$ remain centred Gaussian after reweighting. Then, if a given perturbation $g$ of $\Tilde{\mu}$ is sufficiently confining, it is possible to apply (GCI) to swap the density from $g$ to $\mathbf{h}$.
\begin{definition}
\label{def_domination}
Define
    $$\mathbf{h} := (h_1,h_2)$$
    for some non-negative symmetric quadratic forms $h_1: \Omega \rightarrow \mathbb{R}$, $h_2: \Omega^2 \rightarrow \mathbb{R}$.
    
    We say $\mathbf{h}$ \textbf{dominates} $g$ on $A$ whenever 
    $$\mathrm{e}^{g(x_0,...,x_{T-1}) + \sum\limits_{j=0}^{T-1} \mathbbm{1}_{\gamma(j)=1} h_1(x_j) + \sum\limits_{i \in S(\gamma)} h_2 (x_i,x_{i+1})}$$
    is (QC) on $A^{\gamma}$ for each $\gamma \in \{0,1\}^T$.
\end{definition}

    We conclude with a convenient notation for dominating measures. For any measure $\mu$ on $\Omega$ and $\mathbf{x}:= (x_0,...,x_{T-1})$ define
    $$\mathbf{\mu} ^{\la \mathbf{h}(\gamma) \ra}(\de \mathbf{x}) \propto \mathrm{e}^{\sum\limits_{j=0}^{T-1} \mathbbm{1}_{\gamma(j)=1} h_1(x_j) + \sum\limits_{i \in S(\gamma)} h_2 (x_i,x_{i+1})} \mu_\gamma (\de \mathbf{x})$$
    with
    $$\mu_\gamma := \bigotimes_{j=0} ^{T-1} \mu_{\gamma(j)}$$
    where $\mu_0 := \mu^{\times 2}$ and $\mu_1 := \mu$.
    Also,
    $$\mathbf{\mu} ^{-\la \mathbf{h}(\gamma) \ra} := \mathbf{\mu} ^{\la -\mathbf{h}(\gamma) \ra}. $$
Similarly, for any function $g $ which is bounded from above,
$$\mu^{\la g \ra} (\de x) \propto \mathrm{e}^{g(x)}  \mu (\de x).$$
In this notation it follows that if $\mathbf{h}$ dominates $g$ on $A$, then $\Tilde{\mu}^{\la g \ra} \preceq \Tilde{\mu}^{-\la \mathbf{h}(\gamma)\ra}$ on $A^\gamma$. This observation leads to the second Theorem.

\begin{theorem}
    \label{produt_prop}
    Let $\mu$ be a centred, non-degenerate Gaussian measure.  Let $A \in \mathcal{B}$ be a closed, symmetric convex set with $\mu(A^c) < 0.1$ and $\mu(\partial A)=0$. Fix some $T \in \mathbb{N}$, let $f \in C_b(\Omega^T;\mathbb{R})$ be (QC), and let $g \in C (\Omega^T;\mathbb{R})$ be bounded above with $\mathrm{e}^g$ (QC). If $\mathbf{h}$ dominates $g$ on $c_{1} A$,
    then there exist weights \footnote{We define weights as $w(\gamma) \geq 0$ for all $\gamma \in \{0,1 \}^T $ and $\sum\limits_{\gamma \in \{0,1 \}^T} w(\gamma) = 1$.}  $\lt\{w({\gamma}) : \gamma \in \{0,1\}^T \rt\}$  satisfying
    $$\sum\limits_{\gamma \in \{0,1 \}^T} w(\gamma)|M_1(\gamma )| \geq T(1- \delta)$$
    and
    \begin{equation}
        \label{integral_inequ_product}
        (\mu^{\otimes T})^{\la g \ra} (f) \geq \sum\limits_{\gamma \in \{0,1 \}^T} w (\gamma) \mathbf{\mu} ^{- \la \mathbf{h}(\gamma) \ra}  (f).
    \end{equation}
\end{theorem}

As an application of the developed theory we use Theorem \ref{produt_prop} to derive an upper bound on the mean square displacement of Brownian paths perturbed by a class of pair potentials $W$. 
\begin{assumption}
    \label{weaker_potential}
    $W: \bbR_{\geq 0} \times \bbR_{\geq 0} \rightarrow \bbR $ is bounded from above, jointly continuous and
    $$x \mapsto W(x,t)$$
    is (QC) for each fixed $t >0$. Also,
    there exists $\eps > 0$, $\delta >0$ and a constant $C_{2\eps} > 0$ s.t. 
$$ x \mapsto W(x,t) + C_{2 \eps} x^2$$
is decreasing on $[0,2\eps]$ for each $t \in [0, 2\delta]$. 
\end{assumption}
Let $T>0$ and denote by $\Probm_{[0,T]}$ the law of $d$-dimensional Brownian Motion (BM) on $C\lt([0,T];\mathbb{R}^d \rt)$. We call the measure
\begin{equation}
    \label{perturbed_bm}
\ConcavePotentialMeasure_{\alpha,T} (\de x) \propto \mathrm{e}^{\alpha \int_0 ^T \int_0 ^T W(\Vert x_t-x_s\Vert , |t-s| ) \de t\de s} \Probm_{[0,T]}(\de x)
\end{equation}
 perturbed measure with coupling strength $\alpha > 0$ and pair potential $W$.

% \mscomment{From our meeting: it seemed like a more general assumption could be that $U(x,t)$ is QC in $x$ for any fixed $t$, and moreover $U(x,t)+C_{2\eps} x^2$ is decreasing on $x\in [0,2\eps]$ for any fixed $0\leq t\leq 2$. (By rescaling time, this can be weakened to $0\leq t\leq 2\eps$ if one wants.)}

\begin{theorem}
\label{weaker_pot_thm}
    Let $W$ be s.t. Assumption \ref{weaker_potential} is satisfied. 
    For $T >0 $ and $\alpha > C(W)$, where $C(W)$ is a constant depending on the pair interaction, it holds that
    $$\ConcavePotentialMeasure_{\alpha,[0,T]} \lt( \Vert  x_{0,T}\Vert ^2 \rt) \leq O \lt( \frac{T \log(\alpha)^3}{\alpha} + \lt( \frac{\alpha}{\log(\alpha)^3} \rt)^{-1/2}  \rt).$$
\end{theorem}

At least for quadratic potentials, this result is optimal up to the $\log(\alpha)$ factors.

% \mscomment{Technically in Proposition~\ref{quadratic_computation_prop}, the interaction tends to $-\infty$ as $x\to \infty$, so it is not bounded. If it doesn't require too many changes, it would actually be nice if the main results also allowed interactions like this.}

% \tscomment{I changed the assumption from $C_b$ to continuous and bounded from above. All we need in Theorem \ref{produt_prop} is that $e^g$ is $C_b$ and (QC). The Riemann expansion is anyways justified by the joint continuity of $W$, so no boundedness is needed there I guess. This now includes the case from the Proposition below.}

% \mscomment{OK awesome! Commenting these out so they don't get in the way}

\begin{prop}
\label{quadratic_computation_prop}
    Let $W(x,t)=-x^2 g(t)$ for $g\colon\bbR_{\geq 0}\to\bbR_{\geq 0}$ with $C_g=\int_0^{\infty} t^2 g(t) \de t <\infty$.
    Then 
    \[
    \hat\bbP_{\alpha,[0,T]} \lt(\|x_{0,T}\|^2/d \rt)
    \geq 
    \frac{T}{1+2\alpha C_g}.
    \]
\end{prop}

% \mscomment{Left the proof here for now. Feel free to move it closer to other proofs.}

\begin{proof}
    % Observe that $\hat\bbP_{\alpha,[0,T]}$ is a centered Gaussian measure. 
    By a continuous analog of \cite[Eq. (3.3)]{MV19}, we have the identity
    \[
    \hat\bbP_{\alpha,[0,T]}\lt(\frac{\|x_{0,T}\|^2}{dT}\rt)
    =
    \sup_{f\in C^{\infty}([0,T])}
    \lt\{
    \frac{2(f(T)-f(0))}{\sqrt{T}}
    -
    \int_0^T f'(t)^2 \de t 
    -
    \int_0^T
    \int_0^T 
    \alpha 
    |f(t)-f(s)|^2
    g(|t-s|)
    \de t \de s
    \rt\}.
    \]
    Specializing to linear $f(t)=Lt$, the right-hand side is at least
    \begin{align*}
    \sup_{L\geq 0}\lt\{
    2L\sqrt{T}
    -
    L^2 T
    -
    L^2 \alpha 
    \int_0^T 
    \int_0^T (t-s)^2 g(|t-s|)\de t \de s
    \rt\}
    =
    \frac{T}{T+\alpha \int_0^T 
    \int_0^T (t-s)^2 g(|t-s|)\de t \de s}
    \end{align*}
    which is achieved by 
    \[
    L=\frac{\sqrt T}{T+\alpha \int_0^T 
    \int_0^T (t-s)^2 g(|t-s|)\de t \de s}.
    \]
    Finally since $g$ is non-negative,
    \[
    \int_0^T 
    \int_0^T (t-s)^2 g(|t-s|)\de t \de s
    \leq 
    \int_0^T 
    \int_{-\infty}^{\infty} (t-s)^2 g(|t-s|)\de t \de s
    =
    2T
    \int_{0}^{\infty} u^2 g(u)\de u
    =
    2T C_g.
    \]
    Combining completes the proof.
\end{proof}

% \mscomment{I attempted to clarify the role of unbounded potentials a bit:}

\begin{remark}
\normalfont
    We note that Theorem~\ref{integra_ineq_prop} and \ref{produt_prop} do not apply as stated to the Fr{\"o}hlich Polaron, which involves a Coulomb potential that is unbounded near the origin.
    As shown in e.g. \cite[Proposition 2.5]{Se22}, such potentials can be truncated to become uniformly bounded while approximately preserving the variance of $x_{[0,T]}$.
    However the required truncation will depend on $\alpha$ (and in principle $T$), which enables genuinely different asymptotic behavior.
\end{remark}

\section{Proofs of Theorem \ref{integra_ineq_prop}  and Theorem \ref{produt_prop}}

\subsection{Proof of Theorem \ref{integra_ineq_prop} }
Recall that $\Omega$ is a separable Banach space. It is known that any centred, non-degenerate Gaussian measure $\mu$ on $\Omega$ is the weak limit of centred Gaussian measures with finite-dimensional support. A proof can be found in the Appendix.
\begin{theorem}
    \label{approximation_theorem}
    Let $\mu$ be a centred, non-degenerate Gaussian measure. Then, there exists a sequence of Gaussian measures $(\mu_n)_{n \in \mathbb{N}}$ with supp$(\mu_n) = \Omega_n$ finite-dimensional and 
    $$\mu_n \rightarrow \mu$$
    weakly.
\end{theorem}

For the proof of the next Lemma we also refer to the Appendix.

\begin{lemma}
\label{boundary_lemma}
Let $\mu$ be a Borel probability measure. Moreover, let $A$ be a symmetric, closed convex set s.t. $\mu(A) > 0$ and $\mu(\partial A )=0$. Then, for arbitrary $C > 1$ and $\delta >0$ there exists $C < c < C + \delta$ such that $\mu(\partial c A )=0$.
\end{lemma}

\begin{proof}[Proof of Theorem \ref{integra_ineq_prop}]
    We start with showing (\ref{integral_inequ_1}). Use Theorem \ref{approximation_theorem} to find a sequence $(\mu_n)_{n \in \mathbb{N}}$ such that $\mu_n \rightarrow \mu$. Then,
    $\mu_n(A) \rightarrow \mu(A)$ so eventually $\mu_n (A) > 0.9$. Therefore, for $n$ large enough it is possible to apply Lemma \ref{Se22Lemma31} to get a decomposition
    \begin{equation}
    \label{gaussian_measure_split}
        \mu_n = (1-\delta_n) \nu_n + \delta_n \overline{\nu}_n
    \end{equation}
    with $\delta_n \leq \mu_n (A^c)$. It follows that $\delta_n \rightarrow \delta$ for some $0 < \delta \leq \mu(A^c)$ along some subsequence as $\limsup\limits_{n \rightarrow \infty} \mu_n(A^c) \leq \mu(A^c)$. We abuse notation and take this subsequence as $n$.
    By construction 
    $\frac{\de \nu_n}{\de \mu_n}$
    is (QC)
    and supp($\nu_n$) $\subseteq \const A$, where $\const$ is chosen with Lemma \ref{boundary_lemma} s.t. $\mu(\const A)$ = 0. Clearly
    $$\nu_n (\de x) = \mathbbm{1}_{\const A}(x) \nu_n (\de x)$$
    which yields
    $$\frac{\de  \nu_n}{ \de  \mu_n} = \frac{\de  \nu_n}{ \de  \mu_n} \mathbbm{1}_{\const A}.$$
    Now, 
    $\mu_n (\const A) \rightarrow \mu( \const A)$  implies that 
    $$  \mu_n^{(\mathbbm{1}_{\const A} )}\rightarrow  \mu^{(\mathbbm{1}_{\const A})} .$$
    This is seen by an application of the Portmanteau theorem and 
    %taking any $B$ with $\mu(\partial B)=0$ and noting that \linebreak $\partial (\const A \cap B) \subseteq \partial \const A \cup \partial B$ is a set with measure $0$ under $\mu$. 
    %The claim then follows by 
    $$\mu_n^{(\mathbbm{1}_{\const A} )} (B) = \frac{\mu_n(\const A \cap B)}{\mu_n(\const A)}.$$ 
    
    We deduce that for any $f \in C_b (\Omega; \mathbb{R})$
    \begin{equation}
        \label{weak_conv_restricted_measures}
         \mu_n (f;\const A) \rightarrow  \mu (f;\const A).
    \end{equation}
    We are now ready to show (\ref{integral_inequ_1}). Taking a (QC) function $f \in C_b (\Omega; \mathbb{R}_{\geq 0})$ and using the decomposition of $\mu_n$, it is easily seen that 
    \begin{equation}
        \nonumber
        \begin{split}
            \mu_n (f) &= (1-\delta_n) \nu_n(f)+ \delta_n \overline{\nu_n} (f) \\
            &= (1-\delta_n) \nu_n(f;\const A )+ \delta_n \overline{\nu_n} (f) \\
            &\geq (1-\delta_n) \mu_n (f;\const A) + \delta_n \mu_n ^{\times 2} (f). \\
        \end{split}
    \end{equation}
    By (\ref{weak_conv_restricted_measures}) all expressions on the right hand side converge. The convergence of the left hand side is immediate.

    To show (\ref{integral_inequ_2}), reweight both sides of equation (\ref{gaussian_measure_split}) by $g$. Thanks to (GCI)
    $$\frac{\nu_n (g)}{ \mu_n (g)} \geq 1.$$
    That is, the weight for the ``good part'' of $\mu_n$ only improves. In other words,
    $$\mu_n ^{(g)} = (1- \delta_n) \tfrac{\nu_n(g)}{\mu_n(g)} \nu_n ^{(g)} + \delta_n  \tfrac{\overline{\nu_n}(g)}{\mu_n(g)} \overline{\nu_n}^{(g)}$$
    yields 
    $$ \mu_n ^{(g)} = (1- \delta_n ') \nu_n ^{(g)} + \delta_n ' \overline{\nu_n}^{(g)}$$
    with $\delta_n ' \leq \delta_n \rightarrow \delta$. 
    
    Take a (QC) function $f \in C_b(\Omega ;\mathbb{R})$. Notice that $\mu_n ^{(h)}(f) \geq \mu_n (f) \geq \mu_n ^{\times 2}(f)$ and 
    $$\frac{\de  \nu_n ^{(g)}}{\de  \mu_n ^{(h)}} = \frac{\de  \nu_n ^{(g)}}{\de  \nu_n ^{(h)}} \frac{\de  \nu_n ^{(h)}}{\de \mu_n ^{(h)}} \propto (g/h) \frac{\de  \nu_n}{\de  \mu_n}.$$
    Per assumption $g/h$ is (QC) on the support of $\nu_n$ and we may calculate
    \begin{equation}
        \nonumber
        \begin{split}
            \mu_n ^{(g)} (f) &= (1- \delta_n ')  \nu_n ^{(g)} (f) + \delta_n ' \overline{\nu_n}^{(g)} (f) \\
            &\geq (1- \delta_n ')  \mu_n ^{(h)} (f) + \delta_n ' \mu_n^{\times 2} (f) \\
            &\geq (1- \delta_n)  \mu_n ^{(h)} (f) + \delta_n \mu_n^{\times 2} (f) \\
            &\rightarrow (1- \delta)  \mu ^{(h)} (f) + \delta \mu^{\times 2} (f) \\
        \end{split}
    \end{equation}
    The fact that 
    $\mu_n ^{\times 2} \rightarrow \mu ^{\times 2}$ and $\mu_n ^{(g)} \rightarrow \mu^{(g)}$ follows directly from  $\mu_n \rightarrow \mu$.
\end{proof}

\subsection{Proof of Theorem \ref{produt_prop}}

\begin{lemma}
    \label{tightness_lemma}
    Let $(\nu_n)_{n \in \mathbb{N}}$ and $(\mu_n)_{n \in \mathbb{N}}$ be two sequences of probability measures on $\Omega$ such that for each convex, symmetric set $K$
    $$\nu_n (K) \geq \mu_n(K)$$
    and $\mu_n \rightarrow \mu$ weakly. Then, the sequence $(\nu_n)_{n \in \mathbb{N}}$ is tight.
\end{lemma}
\begin{proof}
    Since $\mu_n \rightarrow \mu$, the sequence $(\mu_n)_{n \in \mathbb{N}}$ is tight.
    Take $\delta > 0$ and let  $K \subseteq \Omega$ be compact s.t. $\sup\limits_{n \in \mathbb{N}} \mu_n(K^c) \leq \delta.$ 
    In a first step we symmetrise $K$ by defining
    $$K_1 := K \cup -K.$$
    This operation clearly preserves compactness. In a second step, take
    $$K_2 := \overline{\text{conv}(K_1)}.$$
    It is known that this set is again compact as $\Omega$ is a Banach space (\cite{AlKi06}, Theorem 5.35). The assumption yields for all $n \in \mathbb{N}$
    $$\nu_n(K_2 ^c) \leq \mu_n (K_2 ^c) \leq \mu_n(K^c)$$
    and so
    $$\sup\limits_{n \in \mathbb{N}} \nu_n(K_2 ^c) \leq \sup\limits_{n \in \mathbb{N}} \mu_n(K_2 ^c) \leq \delta,$$
    showing the claim.
\end{proof}

\begin{proof}[Proof of Theorem \ref{produt_prop}]
    Use Theorem \ref{approximation_theorem} to get a sequence of centred Gaussian measures $(\mu_n)_{n \in \mathbb{N}}$ with finite-dimensional support such that $\mu_n \rightarrow \mu$. By the same argument as in Theorem \ref{integra_ineq_prop}, eventually
    $$\mu_n = (1-\delta_n) \nu_n + \delta_n \overline{\nu}_n$$ is possible with supp($\nu_n$) $\subseteq c_{1} A$ and $\delta_n \leq \mu_n(A^c)$.

    Recall that the measure of interest is $\mu^{\otimes T}$. Doing the split of measures just mentioned in each coordinate leads to weighted measures of type
    $$\Psi_{n,\gamma} = \bigotimes\limits_{i=0}^{T-1} \psi_{n,\gamma(i)}$$
    where $\psi_{n,0} := \overline{\nu}_{n}$ and $\psi_{n,1} := \nu_n$. Indeed, it holds that
    \begin{equation}
        \label{convex_combination_measure}
        \mu_n ^{\otimes T} = \sum\limits_{\gamma \in \{0,1 \}^T} \delta_n ^{ |M_0(\gamma)| } (1-\delta_n)^{|M_1(\gamma)|} \Psi_{n,\gamma}.
    \end{equation}
    Reweight both sides of (\ref{convex_combination_measure}) by $\mathrm{e}^g$ and note that the new weights for a specific $\gamma$ are given by
    $$w_n(\gamma) := \delta_n ^{ |M_0(\gamma)| } (1-\delta_n)^{|M_1(\gamma)|}\frac{\Psi_{n,\gamma}(\mathrm{e}^g)}{\mu_n ^{\otimes T}\lt(\mathrm{e}^g \rt)}.$$
    It then follows that
    $\sum\limits_{\gamma: \gamma(i) = 1} w_n(\gamma) = (1-\delta_n) \frac{\mu_n \otimes ... \otimes \nu_n \otimes... (\mathrm{e}^g)}{\mu_n ^{\otimes T}(\mathrm{e}^g)}$, where the quotient is $\geq 1$ by (GCI). This yields
    $$\sum\limits_{\gamma \in \{0,1 \}^{T}} w_n(\gamma) |M_1(\gamma)| = \sum\limits_{i=0} ^{T-1}  \sum\limits_{\gamma: \gamma(i) = 1} w_n(\gamma) \geq T(1-\delta_n).$$
    Use Lemma \ref{tightness_lemma} to deduce that $(\nu_n)_{n \in \mathbb{N}}$ is tight. 
    Then, there exists a weak limit measure along some subsequence for which $\delta_n \rightarrow \delta \leq \mu(A^c)$ (which we will also denote by $n$). Along this sequence $(\overline{\nu}_n)_{n \in \mathbb{N}}$ also converges weakly to some limit measure since
    $$\mu_n = (1-\delta_n)\nu_n + \delta_n \overline{\nu}_n .$$
    As the underlying space is assumed to be separable, product measures of weakly convergent measures also converge weakly to the product of the respective limits \cite{Bi99}. Therefore, for each $\gamma$, there exists $\Psi_{\gamma}$ s.t. 
    $$\Psi_{n,\gamma} \rightarrow \Psi_{\gamma}$$
    from which it follows that
    along this subsequence all weights $w_n(\gamma)$ 
 converge by construction and are bounded from below as claimed.

    Now,
    \begin{equation}
    \nonumber
        \begin{split}
            \frac{\de \Psi_{n, \gamma} ^{\la g \ra}}{\de  \mu_{n} ^{-\la \mathbf{h}(\gamma) \ra}} &= \frac{\de \Psi_{n, \gamma} ^{\la g \ra}}{\de  \Psi_{n, \gamma}} \frac{\de  \Psi_{n, \gamma}}{\de  \mu_{n,\gamma}} \frac{\de  \mu_{n,\gamma}}{\de  \mu_{n} ^{-\la \mathbf{h}(\gamma) \ra}} \\
            &= \lt( \frac{\de  \Psi_{n, \gamma} ^{\la g \ra}}{\de  \Psi_{n, \gamma}} \frac{\de  \mu_{n,\gamma}}{\de  \mu_{n} ^{-\la \mathbf{h}(\gamma) \ra}} \rt) \frac{\de  \Psi_{n, \gamma}}{\de  \mu_{n,\gamma}}.
        \end{split}
    \end{equation}
    The first term is proportional to
    $$\exp \lt( g(x_0,...,x_{T-1}) + \sum\limits_{j=0}^{T-1} \mathbbm{1}_{\gamma(j)=1} h_1(x_j)  + \sum\limits_{i \in S(\gamma)} h_2 (x_i,x_{i+1}) \rt)$$
     which is (QC) on $(c_1 A)^\gamma =\text{supp}\lt(\Psi_{n, \gamma}\rt)$ by assumption. The second RN-derivative is (QC) due to the fact that each coordinate is. Remembering that $\mu_{n} ^{-\la \mathbf{h}(\gamma) \ra}$ is centred Gaussian,
    an application of (GCI) yields
    \begin{equation}
        \nonumber
        \begin{split}
            (\mu_{n}^{\otimes T}) ^{\la g \ra)} (f) &= \sum\limits_{\gamma \in \{0,1 \}^{T}} w_n (\gamma) \Psi_{n, \gamma} ^{\la g \ra}(f) \\
            &\geq \sum\limits_{\gamma \in \{0,1 \}^{T}} w_n (\gamma) \mu_{n} ^{- \la \mathbf{h}(\gamma) \ra} (f)
        \end{split}
    \end{equation}
By taking the subsequence found before to infinity, the left hand side converges as weak convergence survives under taking products and both $f$ and  $\mathrm{e}^g$ are continuous and bounded. For the right hand side, the same argument holds by recalling that along the chosen subsequence all $w_n (\gamma)$ converge, as well.
\end{proof}

    We want to note that the assumptions in the previous Theorem can be modified. It is only required that $\mathbf{h}$ consists of quadratic, symmetric non-negative forms on the finite dimensional subspaces supp$\lt(\Psi_{n, \gamma}\rt)$. This is because we only use (GCI) for the finite dimensional measures.

    Additionally, we may trade the condition $\mu(\partial A)= 0$ for 
    \begin{equation}
        \nonumber
        \limsup\limits_{n \rightarrow \infty} \mu_n (A^c) < 0.1.
    \end{equation}
    Indeed, if the above condition is satisfied, one may still apply Lemma \ref{Se22Lemma31} to the approximating sequence eventually; in comparison to Theorem \ref{integra_ineq_prop} the convergence of $\mathbbm{1}_{\const A}$ is not required.

\section{Application on Path Space}
The goal of this section is to prove Theorem \ref{weaker_pot_thm}.
The following Corollary is easily deduced from Theorem \ref{produt_prop} by a standard approximation argument.

\begin{corollary}
\label{corollary_msd_2}
    Let the situation in Theorem \ref{produt_prop} be given. In the special case that $\Omega = C \lt([0,1]; \mathbb{R}^d \rt)$ is the path space for some $d \in \mathbb{N}$ and $f: C \lt( [0,1]; \mathbb{R}^d \rt) \rightarrow \mathbb{R}^d $ is linear,
    $$ \lt( \mu^{\otimes T} \rt)^{(g)} \lt(\Vert f(x) \Vert ^2 \rt) \leq \sum\limits_{\gamma \in \{0,1 \}^T} w (\gamma) \mathbf{\mu} ^{- \la \mathbf{h}(\gamma) \ra} \lt(\Vert f(x) \Vert ^2 \rt).$$
\end{corollary}
Let $c > 0$ be arbitrary. We will choose a specific $c$ shortly. 
In this section take
$\Omega = C \lt([0,c];\mathbb{R}^3 \rt)$. Whenever convenient identify $C \lt([0,c];\mathbb{R}^3 \rt)^2 \simeq C \lt([0,2c];\mathbb{R}^3 \rt)$.
Write
    $$K_{R} := \lt\{ (x_{t})_{t\leq c} \in C \lt([0,c]; \mathbb{R}^3 \rt) : \sup\limits_{s \in [0,c]} \Vert x_{0,s}\Vert  \leq R \rt\}.$$ 

It is easily seen that  $\Probm_{[0,c]}\lt(\partial K_R \rt) = 0$ for any $R>0$. 
%\mscomment{I agree the statement is easy, but I do not actually see why it follows from the reflection principle. (I think I would agree if we were in $1$-dimension.)}
%\tscomment{Yep, my bad! In 1-d it is definitely correct, but in 3-d another basic argument suffices. }
In the following we fix $W$ such that Assumption \ref{weaker_potential} is satisfied. We choose the corresponding $\eps >0$ and $\delta>0$ accordingly.

\begin{comment}
   As an example, consider $W(x,t) = U(x) v(t)$ with $v(t) = \mathrm{e}^{-t}$
for some fixed $A>0$ take $U$ as 
    $$x \mapsto \begin{cases}
A- A^2 x &\text{ if $x \in [0, 1/A]$},\\
\tfrac{1}{x} &\text{ if $r \in (1/A,\infty)$},
\end{cases}$$
which was used in \cite{Se22} to approximate the Coloumb potential.
We could also take
$$U(x) := 1 - x^2 \mathbbm{ 1}_{[0,1]} (x),$$
which shows that the confining behaviour of $W$ around $0$ is the most important feature.   
\end{comment}

Take $c_b >0$ s.t. for all $\alpha \geq 2$
$$\Probm_{[0,1]} \lt( \sup\limits_{s \leq 1} \Vert x_s \Vert \geq c_b \sqrt{\log(\alpha)} \rt) \leq \alpha^{-10}.$$
It is possible to change the diffusion radius and still obtain the same  bound as in the last display if we restrict the path length.
Diffusive scaling of BM
$$\frac{1}{\sqrt{c}} x_{ct} \sim x_t \iff x_{ct} \sim \sqrt{c} x_t$$
implies that
\begin{equation}
    \label{endpoint_equ}
    \Probm_{[0,c(\alpha)]} \lt( \sup\limits_{s \leq c(\alpha)} \Vert x_s \Vert \geq \eps/2 \rt) \leq \alpha^{-10}
\end{equation}
with
$$c = c (\alpha) := \frac{\eps^2}{4 c_b^2 \log(\alpha)}.$$

Fix $\alpha \geq 2$ and let $T > 0$ be a time s.t. $\tfrac{1}{c} T \in \mathbb{N}$. One can scale the BM in the underlying measure so that this condition is not restrictive. Split $[0,T]$ in $\tfrac{1}{c} T$ smaller intervals of size $c$. By  (\ref{endpoint_equ}) we know that on each of these small intervals BM is essentially concentrated on $K_\eps$; i.e.,
$$\Probm_{[0, c]} \lt( K_{\eps/2} ^c \rt) \leq \alpha^{-10}.$$
Define the vector $\mathbf{h}^\beta$ with $\beta = (\beta_0,\beta_1)$ by
$$h_0 ^{\beta_0}: (x_t)_{t \in [0,c]} \mapsto - \beta_0 \int_0 ^{c} \int_0 ^{c} \Vert x_t - x_s\Vert ^2 \de t \de s,$$
$$h_1 ^{\beta_1} : (x_t)_{t \in [0,2c]} \mapsto - \beta_1 \int_0 ^{c} \int_{c} ^{2c} \Vert x_t - x_s\Vert ^2 \de t \de s.$$
It should be noted that both functions are symmetric, quadratic non-negative forms on path space. 

Our next goal is to show that $\alpha W$ is dominated by $\mathbf{h}^{\beta}$ on $K_{\eps}$ for certain parameters $\beta$, considered on the product space $\Omega^{\otimes\tfrac{1}{c}T}$. To do this, let $\lfloor x \rfloor$ be the biggest number $l< x$ with $l = c n$ for some $n \in \mathbb{N}$. Also, for $l \in \mathbb{N}$, $\gamma(l c) = \gamma(l)$.

\begin{prop}
    \label{densityreplacement}
    Let $T \frac{1}{c} \in \mathbb{N}$ and $\gamma \in \{0,1\}^{T \frac{1}{c}}$.
    Define 
    $$F(s,t) := \begin{cases}
        \beta_0 \text{  if $\lfloor s \rfloor$= $\lfloor t \rfloor$ and $\gamma(\lfloor s \rfloor)=1$}, \\
        \beta_1 \text{  if $\lfloor s \rfloor \in S(\gamma)$ and $t \in \big[ \lfloor s \rfloor + c , \lfloor s \rfloor +2 c \big]$}, \\ 
        0 \text{   else.}
    \end{cases}$$
     For $0 \leq \beta_0, \beta_1 \leq \alpha C_{2\eps}$ and $c = c(\alpha) < \delta$, the map
\begin{equation}
    \nonumber
    \begin{split}
        & x \mapsto 
         \exp \lt( \int_0 ^T \int_0 ^T \alpha W(\Vert x_t - x_s\Vert,|t-s|) + F(s,t) \Vert x_t - x_s\Vert ^2 \de t \de s \rt)
    \end{split}
\end{equation}
    is (QC) on $K^\gamma _{\eps}$. It follows that $\mathbf{h^\beta}$ dominates 
    $$g(x) := \int_0 ^T \int_0 ^T \alpha W(\Vert x_t - x_s\Vert,|t-s|) \de t \de s$$
    on  $K_{ \eps}$.
\end{prop}

\begin{proof}
The proof is given in Example \ref{integral_example} by noting that the function $F$ was chosen in such a way that
$$x \mapsto 
           \alpha W(\Vert x_t - x_s\Vert ,|t-s|) + F(s,t) \Vert x_t - x_s\Vert ^2  $$
is symmetric and quasi-concave on $K_{\eps} ^{\gamma}$. This is because only nearest neighbor intervals are considered (whenever fluctuations are both bounded). That is, for all $s,t \in [0,T]$ with $F(s,t) \neq 0$ we have
$|t-s| \leq 2\delta$ (for sufficiently big $\alpha$) and $\Vert x_t - x_s\Vert \leq 2 \eps$, which implies per assumption that the map in the above display is decreasing. The assumption that $W$ is jointly continuous justifies the approximation by Riemann sums.
\end{proof}

In the following we write  
$c$ used in the definition of $\mathbf{h}^\beta$ in the subscript of $\Probm$. It is very helpful to think of $c$ as the basic block size in which the perturbed measure is eventually split. It is therefore required that the entire time interval $T$ is a multiple of the block size $c$.
Define for $s \in \mathbb{N}$ and $c > 0$
\begin{equation}
\label{prod_measure_rewrite}
    \begin{split}
        \Bar{\Probm}_{c, [0,c s]} ^{\beta} (\de x) & \propto  \mathrm{e}^{- \beta \sum\limits_{j=0} ^{s-1} \int_{c j} ^{c(j+1)} \int_{cj} ^{c(j+1)} \Vert x_t-x_s\Vert ^2 \de t\de s - \beta \sum\limits_{j=0} ^{s-2} \int_{c j} ^{c(j+1)} \int_{c(j+1)} ^{c (j+2)} \Vert x_t-x_s\Vert ^2 \de t\de s} \Probm_{[0,c s]}(\de x) \\
        &\propto \mathrm{e}^{- \beta \sum\limits_{j=0} ^{s-1} \int_{cj} ^{c(j+1)} \int_{cj} ^{c (j+1)} \Vert x_t - x_s\Vert ^2 \de t\de s - \beta \sum\limits_{j=0} ^{s-2} \int_{c j} ^{c(j+1)} \int_{c(j+1)} ^{c(j+2)} \Vert x_t+ x_{c(j+1)} - x_s \Vert ^2 \de t\de s} \Probm_{[0,c]}^{\otimes s}(\de x).
    \end{split}
\end{equation}

The following result is deduced in the same manner as in \cite[Section 6]{Se22}.
\begin{lemma}[{\cite[Lemma 6.5]{Se22}}]
Let $s \in \mathbb{N}$ and $\beta \geq 2$. Then,
\label{long_time_fluctuation_lemma}
$$\Bar{\Probm}_{1,[0,s]} ^{\beta} \lt(\Vert x_{0,s}\Vert ^2 \rt) \leq O \lt( \frac{s}{\beta} + \beta^{-1/2} \rt).$$
\end{lemma}
By an application of Brownian rescaling one obtains
\begin{equation}
    \label{good_block_calc_weak}
    \Bar{\Probm}_{c,[0,c s]} ^{\beta} \lt(\Vert x_{0,c s}\Vert ^2 \rt) = c \Bar{\Probm}_{1,[0,s]} ^{\beta c^3} \lt(\Vert x_{0,s}\Vert ^2 \rt).
\end{equation}

    Take $\gamma \in \{0,1\}^T$. We define \textbf{good blocks} to be maximal subintervals $[i,j]\subseteq [0,T]$ for which $\gamma(i) = \gamma(i+1) = ... = \gamma(j) =  1$, i.e. also $\gamma(i-1) = 0 = \gamma(j+1)$ (with convention $\gamma(-1) = \gamma(T+1) = 0$). Let $B(\gamma)$ be the set of good blocks associated to $\gamma$, where each block is identified with its length. A \textbf{bad block} is an interval with $\gamma(k) = 0$.

    The next calculation will apply Theorem \ref{produt_prop} to $\ConcavePotentialMeasure_{\alpha,[0,T]}$. By joining the endpoints of the path from previous intervals to the next one it is possible to write $\ConcavePotentialMeasure_{\alpha,[0,T]}$ as a reweighted product measure. As the notation is quite cumbersome, we omit the details. For a simple example we refer to equation (\ref{prod_measure_rewrite}).
    
\begin{proof}[Proof for Theorem \ref{weaker_pot_thm}]

W.l.o.g. let $\tfrac{1}{c}T \in \mathbb{N}$. Choose $\beta = \beta_0 = \beta_1 = \alpha C_{2\eps}$. For fixed $\gamma$ there are $|M_0 (\gamma)| +1$ good blocks at most. By independent increments of Brownian Motion, the structure of $\Probm_{c,[0,T]}^{- \la \mathbf{h}^{\beta} (\gamma) \ra}$ and the fact that the law of $\Probm_{c,[0,T]}^{- \la \mathbf{h}^{\beta} (\gamma) \ra}$ on a good block of length $s$ is given by
 $\Bar{\Probm}_{c,[0,c s]} ^{\beta}$, we split the integral along its blocks to obtain
    $$\Probm_{c,[0,T]}^{- \la \mathbf{h}^{\beta}(\gamma) \ra} \lt(\Vert x_{0,T}\Vert ^2 \rt) \leq \sum\limits_{l \in B(\gamma)} \Bar{\Probm}_{c[0,c l]} ^{\beta} \lt(\Vert x_{0,c l}\Vert ^2 \rt) + |M_0 (\gamma)| \Probm_{[0,c]}^{\times 2}\lt(\Vert  x_{0,c} \Vert ^2 \rt).$$
    Use Proposition \ref{densityreplacement} to apply Corollary \ref{corollary_msd_2} with
    
    \begin{itemize}
        \item  $\mu = \Probm_{[0,c]}$,
        \item $\mathbf{h} = \mathbf{h}^{\beta}$ with $\beta = (\alpha  C_{2 \eps},\alpha  C_{2 \eps})$, 
        \item $g(x) = \alpha \int_0 ^T \int_0 ^T W(\Vert x_t-x_s\Vert ,|t-s|) \de t \de s$,
        \item $A= K_{ \eps}$,
        \item $\Tilde{T} = \tfrac{1}{c} T$
    \end{itemize}
   and estimate with $\delta \leq \alpha^{-10}$ and (\ref{good_block_calc_weak})
\begin{equation}
    \nonumber
    \begin{split}
        & \ConcavePotentialMeasure_{[0,T]}\lt( \Vert  x_{0,T}\Vert ^2 \rt) \leq \sum\limits_{\gamma \in \{0,1 \}^{c^{-1}T}} w (\gamma) \mathbf{\Probm}_{c,[0,T]} ^{- \la \mathbf{h}^{\beta}(\gamma) \ra}  \lt(\Vert  x_{0,T}\Vert ^2 \rt) \\
        &\leq \sum\limits_{\gamma \in \{0,1 \}^{c^{-1}T}} w (\gamma) \Big[ \sum\limits_{l \in B(\gamma)} \Bar{\Probm}_{c,[0,c l]} ^{\beta} \lt(\Vert x_{0,c l}\Vert ^2 \rt) + |M_0 (\gamma)| \Probm_{[0,c]}^{\times 2} \lt(\Vert  x_{0,c} \Vert ^2 \rt) \Big] \\
        &\leq \sum\limits_{\gamma \in \{0,1 \}^{c^{-1}T}} w (\gamma) \lt[  O \lt( \frac{T}{c^3 \beta} + \lt(|M_0 (\gamma)|+1 \rt) (\beta c^3)^{-1/2}  \rt) + |M_0 (\gamma)| \Probm_{[0,1]}^{\times 2} \lt(\Vert  x_{0,1} \Vert ^2 \rt) \rt] \\
        &\leq O \lt( \frac{T}{\beta c^3} + \lt(\alpha^{-10} T +1 \rt) (\beta c^3)^{-1/2} \rt) + \sum\limits_{\gamma \in \{0,1 \}^{c^{-1}T}} w (\gamma) |M_0 (\gamma)| \Probm_{[0,1]}^{\times 2} \lt(\Vert  x_{0,1} \Vert ^2 \rt) \\
        &\leq O \lt( \frac{T}{c^3 \beta} + \lt(\alpha^{-10} T +1 \rt) (\beta c^3)^{-1/2} \rt) + \alpha^{-10} \frac{T}{c} \Probm_{[0,1]}^{\times 2} \lt(\Vert  x_{0,1} \Vert ^2 \rt) \\
        &= O \lt( \frac{T \log(\alpha)^3}{\alpha} + \lt( \frac{\alpha}{\log(\alpha)^3} \rt)^{-1/2}  \rt).
\qedhere
    \end{split} 
\end{equation}
    
\end{proof}

\begin{appendix}

\section{Gaussian Approximation}
We will require the notion of an abstract Wiener Space. For us, an abstract Wiener Space is a triplet which consists of a separable Banach space $\Omega$, a non-degenerate centred Gaussian $\mu$ on $\Omega$ and the corresponding Cameron Martin space $H$. Recall that $H$ is dense in $\Omega$, uniquely determined by $\mu$ and that $H \hookrightarrow \Omega$ is continuous (\cite{St23}, Theorem 3.3.2.).

The continuity of the embedding implies via the Riesz Representation Theorem that for each $\xi \in \Omega^*$ there exists a corresponding $h_\xi \in H$ s.t. 
$\xi(h) = \la h_\xi,h \ra_H$ for each $h \in H$. Moreover, one can show that the map
 $$\xi \mapsto h_\xi$$
is linear, continuous and one-to-one (\cite{St23}, Lemma 3.3.1 \textbf{(i)}). Also,
    for each $f \in \Omega^*$ it holds that (\cite{St23}, Theorem 3.3.2)
    \begin{equation}
        \label{cameron_martin_cf}
        \int \mathrm{e}^{i f(x)} \mu (\de x) = \mathrm{e}^{-\frac{1}{2} \Vert h_f \Vert_H ^2}.
    \end{equation}
At last, it can be concluded that there exists a sequence $\{ \xi_n : n \in \mathbb{N}\}\subseteq \Omega^*$ s.t. the set  $\{ h_{\xi_n} : n \in \mathbb{N}\}$ is an ONB for $H$ (\cite{St23}, Lemma 3.3.1 \textbf{(iii)}).
It follows that for each $h \in H$
$$h = \sum\limits_{i=1} ^\infty \la h_{\xi_i}, h \ra_H h_{\xi_i} = \sum\limits_{i=1} ^\infty \xi_i (h) h_{\xi_i}.$$

Define for each $x \in \Omega$ the projection

$$P_n x := \sum\limits_{i=1} ^n \xi_i (x) h_{\xi_i}.$$
It is easily seen that $P_n$ extends the projection of $H$ onto span$\{h_{\xi_1},...,h_{\xi_n}\}$ to $\Omega$.

We define
    $$\mu_n := \mu \circ (P_n)^{-1}.$$
By definition it is clear that supp$(\mu_n) \subseteq \Omega$ is finite-dimensional. Moreover, it is also immediate that $\mu_n$ is centred Gaussian. To find its covariance take $f \in \Omega^*$ and calculate
    \begin{equation}
        \nonumber
        \begin{split}
            \int \mathrm{e}^{i f(x)} \mu_n(\de x) &= \int \mathrm{e}^{i \sum\limits_{j=1} ^n \xi_j(x) f(h_{\xi_j})} \mu(\de x) \\
            &= \exp\lt( -\frac{1}{2}  \lt\la \sum\limits_{j=1} ^n h_{\xi_j} f( h_{\xi_j}), \sum\limits_{j=1} ^n h_{\xi_j} f( h_{\xi_j})  \rt\ra_H \rt) \\
            &= \exp\lt( -\frac{1}{2} \sum\limits_{j=1} ^n \la h_f, h_{\xi_j} \ra_H ^2 \rt). 
        \end{split}
    \end{equation}

It can be concluded that for $f \in \Omega^*$
\begin{equation}
    \label{ito_nisio_1}
    \int \mathrm{e}^{i f(x)} \mu_n(\de x) \rightarrow \int \mathrm{e}^{i f(x)} \mu(\de x)
\end{equation}
via (\ref{cameron_martin_cf}).
Define the random variables
$$X_j: \Omega \rightarrow \Omega, \: \: \: x \mapsto \xi_j (x) h_{\xi_j}.$$
Notice that $X_i$ is independent of $X_j$ whenever $i \neq j$ and that $X_i$ is symmetrically distributed because $\xi_i$ is. Then, using the Itô-Nisio Theorem (\cite{ItNi68}, Theorem 4.1), we see that (\ref{ito_nisio_1}) implies $\mu$-a.e.
$$ \sum\limits_{j=1} ^n X_j \rightarrow I,$$
where $I$ is the identity on $\Omega$.
We collect this result in the following Lemma.

\begin{lemma}
    \label{main_theorem}
    Let $\Omega$ be a separable Banach space endowed with Borel-$\sigma$-field $\mathcal{B}$ and let $\mu$ be a non-degenerate, centred Gaussian measure. Then, there exists a sequence of projections $(P_n)_{n \in\mathbb{N}}$ such that $\Omega_n := P_n(\Omega)$ is finite-dimensional and $\mu$-a.e. 
    $$\lim\limits_{n \rightarrow \infty} P_n x = x.$$
\end{lemma}

\begin{prop}
    $\mu_n \rightarrow \mu$ weakly.
\end{prop}

\begin{proof}
    Let $A \in \mathcal{B}$ be an open set and let $(x_n)_{n \in \mathbb{N}} \subseteq \Omega$ be a sequence s.t. $x_n \rightarrow x$. We claim that 
    $$\mathbbm{1}_A (x_n) \geq \mathbbm{1}_A (x)$$
    for all $n \geq n_0 \in \mathbb{N}$ which implies that
    $$\liminf\limits_{n \rightarrow \infty} \mathbbm{1}_A (x_n) \geq \mathbbm{1}_A (x).$$
    Indeed, if $x \in A$, there exists some ball with radius $\eps >0$ centred at $x$ that is also contained in $A$. $x_n$ converges to $x$, so eventually $x_n \in B(x,\eps) \subseteq A$ for all $n \geq n_0$. The case $x \notin A$ is trivial. An application of Fatou's Lemma in conjunction with Lemma \ref{main_theorem} (use $x_n := P_n x)$ yields
    \begin{equation}
        \begin{split}
            \liminf\limits_{n \rightarrow \infty} \mu_n(A) = \liminf\limits_{n \rightarrow \infty} \int \mathbbm{1}_A (x) \mu_n(\de x)  &\geq  \int \liminf\limits_{n \rightarrow \infty} \mathbbm{1}_A (P_n x) \mu(\de x) \\
            &\geq \int \mathbbm{1}_A (x) \mu(\de x) = \mu(A). 
        \end{split}
    \end{equation}
    Using the Portmanteau Theorem shows the claim.
\end{proof}

\section{Measure of Boundary of Symmetric, Convex Sets}
We use $\sqcup$ to denote disjoint unions, and always assume $\mu$ to be a Borel probability measure on a real Banach space.
\begin{prop}
\label{ball_prop}
    Let $\mu$ be a Borel probability measure. Moreover, let $A$ be a closed, symmetric and convex set s.t. $\mu(A) > 0$ and $\mu(\partial A )=0$. Then, there exists $\eps > 0$ such that $B(0,\eps) \subseteq A$.
\end{prop}

\begin{proof}
    Since $A= \partial A \sqcup A^\circ$, the assumption gives that $\mu(A^\circ) > 0$. Then, $A^\circ \neq \emptyset$ and so there exists $x \in A^\circ$ with $B(x,\eps) \subseteq A$. By symmetry of $A$,  $B(-x,\eps) \subseteq A$. Using the convexity we know that  $\text{conv} \big(B(x,\eps),B(-x,\eps) \big) \subseteq A$ and so $B(0,\eps) \subseteq A$.
\end{proof}
\begin{prop}
    \label{line_prop}
    If $A$ is convex,  $x \in A^{\circ}$, $y \in \overline{A}$, then
    $$[x,y) \subseteq A^{\circ}.$$
\end{prop}
\begin{proof}
    We follow the proof given in \cite{HuWe20}. Take $x \in A^{\circ}$ and $y \in \overline{A}$. For $\alpha \in (0,1)$\footnote{The case $\alpha=1$ is trivial.} write $z = \alpha x + (1-\alpha)y$. We wish to show that $z \in A^{\circ}$. To this end, take $(y_n)_{n \in \mathbb{N}} \subseteq A$ with $y_n \rightarrow y$. Let $\eps >0$ so that $B(x,\eps) \subseteq A$. Note that the sequence $(x_n)_{n \in \mathbb{N}}$ defined by
    $$x_k := \frac{1}{\alpha} \big(z - (1-\alpha)y_k \big)$$
    converges to $x$. That is, there exists $n_0 \in \mathbb{N}$ s.t.,for all $n \geq n_0$, $x_n \in B(x,\eps)$ and so
    \begin{equation}
        \nonumber
        \begin{split}
            x_k \in B(x,\eps)
            \iff z-(1-\alpha)y_k \in \alpha B(x,\eps) 
            \iff z \in \alpha B(x,\eps) + (1-\alpha)y_k.
        \end{split}
    \end{equation}
    By convexity $\alpha B(x,\eps) + (1-\alpha)y_k \subseteq A$, which shows the claim.
\end{proof}
\begin{prop}
\label{contain_prop}
    Let $A$ be an open, symmetric and convex set. Moreover, assume that $B(0,\eps) \subseteq A$ for some $\eps > 0$. Then, for any $t >1$, 
    $$\overline{A} \subseteq t A.$$
\end{prop}

\begin{proof}
    Write $\overline{A} = \partial A \sqcup A$. $A \subseteq tA$ is obvious by $0 \in A$ and the convexity. To show the inclusion for $\partial A$ define for $p > 0$ and $x \in B$ the scaled convex cone
    $$M (p,x) :=  \text{ conv}\big(B(0,p\eps),px \big)$$
    so that $M(1,x)\backslash x \subseteq A$ whenever $x \in \partial A$ (use Proposition \ref{line_prop}). Notice that 
    $M(1,x) \subseteq M(t,x) \backslash tx$ for $t>1$\footnote{The only problematic point could be $x$, but $[0,x) \subseteq A$, so $x \in [0,tx) \subseteq tA$ as $t>1$.}. 
    But now $$M(1,x) \subseteq M(t,x) \backslash tx \subseteq tA$$ for $x \in \partial A$. In particular $x$ is now an inner point of $tA$, which shows the claim.
\end{proof}

\begin{proof}[Proof of Lemma \ref{boundary_lemma}]
    Because $A$ is symmetric and convex, for $t>1$ it holds that $tA \supseteq A$ and so the map
    $$t \mapsto \mu(tA^{\circ})$$
    is monotone increasing. It is well known that a monotone increasing function is continuous, up to a countable set of discontinuities. Then, we can find $ c \in (C, C+ \delta)$ which is a point of continuity for the map just introduced.
    Write $cA = \partial cA \sqcup (cA)^{\circ}$.
    Use Proposition \ref{ball_prop} with $A$ to deduce that $A^\circ$ contains $B(0,\eps')$ for some $\eps' >0$. This ball is then also contained in $(cA)^{\circ} \supseteq A^{\circ}$. Obtain with Proposition \ref{contain_prop} and the fact that $cA$ is convex that for any $\eps >0$
    $$(c+\eps)A^{\circ} = (1+ \tfrac{\eps}{c})(cA)^{\circ} \supseteq \overline{(cA)^\circ} = cA$$ and so $$\partial cA \subseteq (c+\eps)A^{\circ} \backslash cA^{\circ}.$$
    We conclude that for $\eps >0$ arbitrary
    $$\mu(\partial cA) \leq \mu( (c+\eps)A^{\circ} ) - \mu(cA^{\circ}).$$
    Since $c$ is assumed to be a point of continuity the claim follows.
    \end{proof}

\end{appendix}

% \clearpage


\begin{thebibliography}{99}

\bibitem{AlKi06} 
	C.D. Aliprantis and K.C. Border.
	\newblock{Infinite Dimensional Analysis: A Hitchhiker's Guide.}
	 Springer, 2006.


\bibitem{BP22a}
V. Betz and S. Polzer,
\newblock{A Functional Central Limit Theorem for Polaron Path Measures.}
\newblock{\em Communications on Pure and Applied Mathematics (2022)}, doi 10.1002/cpa.22080

\bibitem{BP22b}
V. Betz and S. Polzer,
\newblock{Effective Mass of the Polaron: a lower bound.}
\newblock {\em Communications in Mathematical Physics}, 399(1):178 -- 188,
	2023.


\bibitem{Bi99}
    P. Billingsley.
    \newblock{Convergence of Probability Measures.}
    \newblock{\em John Wiley \& Sons Inc.} 1999.

%\bibitem{Bo21}
%    V. I. Bogachev.
%    \newblock{On Approximation of Measures
%by Their Finite-Dimensional Images}
%\newblock{\em Functional Analysis and Its Applications,  Vol. 55, No. 3, pp. 236–241}, 2021.


\bibitem{BrSe22}
	M. Brooks and R. Seiringer.
	\newblock The Fröhlich Polaron at Strong Coupling -- Part II: Energy-Momentum Relation and Effective Mass.
	\newblock {\em preprint: arXiv:2211.03353}, 2022	



\bibitem{DySp20}
	W. Dybalski and H. Spohn.
	\newblock Effective mass of the Polaron{\textemdash}revisited.
	\newblock {\em Annales Henri Poincar{\'{e}}}, 21(5):1573--1594, 2020.

\bibitem{HuWe20} 
	D. Hug and W. Weil.
	\newblock{Lectures on Convex Geometry.}
	\newblock{\em Graduate Texts in Mathematics. Volume 286.} Springer, 2020.


\bibitem{ItNi68} 
	K. Itô and M. Nisio.
	\newblock{On the convergence of sums of independent Banach space valued random variables.}
	\newblock{\em Osaka J. math. 5 (1) 35-48,} 1968.



\bibitem{LP48}
	L.~D. Landau and S.~I. Pekar
	\newblock Effective Mass of a Polaron
	\newblock {\em Zh. Eksp. Teor. Fiz. 18, 419–423,} 1948

\bibitem{LBH-Book}
J. L\"orinczi, F. Hiroshima and  V. Betz. 
{\em Feynman-Kac-Type Formulae and Gibbs Measures}, Berlin, Boston: De Gruyter, 2020. https://doi.org/10.1515/9783110330397


%\bibitem{LuMiPa15} 
%	A. Lunardi, M. Miranda and D. Pallara
%	\newblock{Infinite Dimensional Analysis.}
%	\newblock{\em 19th Internet Seminar.} 2015.



\bibitem{MV19}
C. Mukherjee and S.~R.~S. Varadhan.
\newblock Identification of the polaron measure I: Fixed coupling regime and
  the central limit theorem for large times.
\newblock {\em Communications on Pure and Applied Mathematics}, 73(2):350--383,
  August 2019.





\bibitem{Ro14}
    T. Royen.
    \newblock{A Simple Proof of the Gaussian Correlation Conjecture Extended to Some Multivariate Gamma Distributions.}
    \newblock{\em Far East Journal of Theoretical Statistics, 3:139-145.} 2014.

\bibitem{Se22} 
	M.\ Sellke.
	\newblock{Almost Quartic Lower Bound for the Fröhlich Polaron's Effective Mass via Gaussian Domination}
	\newblock{arXiv:2212.14023 [math-ph]}


\bibitem{Sp87}
	H. Spohn.
	\newblock Effective mass of the polaron: A functional integral approach.
	\newblock {\em Annals of Physics}, 175(2):278--318,  1987.


\bibitem{St23} 
	D. W. Stroock.
	\newblock{Gaussian Measures in Finite and Infinite Dimensions.}
	\newblock{\em Universitext.} Springer, 2023.



 



\end{thebibliography}
\end{document}